\documentclass{article}
\usepackage{amsfonts}
\usepackage{amsmath}
\usepackage{amssymb}
\usepackage{amstext}
\usepackage{amsthm}
\usepackage{amsbsy,upref}
\usepackage{graphicx}

\allowdisplaybreaks
\numberwithin{equation}{section}
\numberwithin{figure}{section}

\theoremstyle{plain}
\newtheorem{theorem}{ Theorem}[section]

\newtheorem{lemma}[theorem]{Lemma}
\newtheorem{corollary}[theorem]{Corollary}

\newtheorem{remark}[theorem]{Remark}
\newtheorem{definition}[theorem]{Definition}

\newcommand{\capas}{\mathop{\rm cap}\nolimits}

\newcommand{\cH}{{\mathcal{H}}}
\newcommand{\cU}{{\mathcal{U}}}

\begin{document}

\title{Spectral gaps for the linear  surface wave model in  periodic channels}

\author{F.L. Bakharev
, K. Ruotsalainen, J. Taskinen}

\maketitle

{\it 
Chebyshev Laboratory, St. Petersburg State University, 14th Line, 29b, Saint Petersburg, 199178 Russia\footnote{The first named author was supported by the St. Petersburg State University
grant  6.38.64.2012 as well as by the Chebyshev Laboratory - RF Government grant 11.G34.31.0026 and by JSC ``Gazprom Neft''. The first and third named authors were also supported by the Academy of Finland project "Functional analysis and applications".}

University of Oulu, Department of Electrical and Information 
Engineering,
Mathematics Division, P.O. Box 4500, FI-90401 Oulu, Finland 

University of Helsinki, Department of 
Mathematics and Statistics,
P.O. Box 68, FI-00014 Helsinki, Finland. }

\begin{abstract} 
We consider the linear water-wave problem in a periodic channel which consists of infinitely many identical containers connected with apertures of width $\epsilon$. Motivated by applications to surface wave propagation phenomena,
we study the band-gap structure of the 
continuous spectrum. We show that  for small apertures there exists a large number of gaps and also find asymptotic formulas for the position of the gaps as $\epsilon \to 0$: the endpoints are determined within corrections of order $\epsilon^{3/2}$. The width of the first bands is shown
to be $O(\epsilon)$. Finally, we give a sufficient condition which guarantees that the spectral bands do not degenerate into eigenvalues of 
infinite multiplicity. 
\end{abstract}


\bigskip

\section{Introduction}
\label{sec1}
\subsection{Overview of the results}
\label{sec:1.1}

Research on wave propagation phenomena in periodic media has been very active during  many decades. The topics and applications include for example photonic crystals, meta-materials, Bragg gratings of surface plasmon polariton waveguides, energy harvesting in piezoelectric materials as 
well as surface wave propagation in periodic channels, which is the subject of this paper. A standard mathematical approach consists of linearisation and posing a spectral problem for an elliptic,
hopefully self-adjoint, equation or system.

Early on it was noticed that waves propagating in periodic media have spectra with allowed bands separated by forbidden frequency gaps. This phenomenon was first discussed by Lord Rayleigh \cite{LordRa}.  It has also attracted some interest in coastal engineering because it provides a possible means of protection against wave damages \cite{Mat,Mei}, for example by varying the bottom topography by periodic arrangements of sandbars. The existence of forbidden frequencies is conventionally related to Bragg reflection of water waves by periodic structures. 
Here, Bragg reflection  is an enhanced reflection which occurs when the wavelength of an incident surface wave is approximately twice the wavelength of the periodic structure. This mechanism works, if the waves are relatively long so that the depth changes can effect them \cite{Mei}.

A similar phenomenon may also happen, when waves are propagating along a channel with periodically 
varying width. In \cite{Liu}, and later \cite{McKee}, the authors studied a channel, the wall of  which had a periodic stepped structure. Using  resonant interaction theory they were able to verify that significant wave reflection could occur. These results  are  based on the assumption of  small wall irregularities.

Gaps in the continuous spectrum for equations or  systems in unbounded waveguides have been studied in many papers, and we refer to \cite{Kuch} for an introduction to the topic. 
In \cite{NaRuTa} the authors studied the linear elasticity system and proved the existence of arbitrarily (though still finitely) many gaps, the number of them depending on a small geometric parameter;  the approach  is similar to  Section 3.1, below, and the result is analogous to Corollary \ref{cor3.2}.
In the setting of the linear water-wave problem, spectral gaps have been studied
in \cite{lin}, \cite{McI}, \cite{na460} and \cite{CaMI}, though the
point of view  is different from  the present work.

In this paper we consider  surface wave propagation using the linear water wave
equation with spectral Steklov boundary condition on the free water surface, see the equations \eqref{problem}--\eqref{b-cond}, which are called the original problem here. The water-filled domain $\Pi^\epsilon$ forms an unbounded periodic channel consisting of infinitely many identical bounded containers connected by apertures of width $\epsilon >0$, see
Figure \ref{fig1}. The first results, Theorem \ref{th3.1} and Corollary \ref{cor3.2} show 
that the essential spectrum $\sigma$ of the original problem (which is expected to be non-empty due to the unboundedness of the domain) has gaps, and the number of them can be made arbitrarily large depending on the parameter $\epsilon$. An explanation of this phenomenon  can be outlined rather simply using the 
Floquet-Bloch theory, though a lot of technicalities will eventually be involved. Namely, if $\epsilon = 0$, the domain becomes a disjoint union of infinitely many bounded containers, and the water-wave problem reduces to a problem on a bounded domain (we call it the limit problem), hence it has a discrete spectrum consisting of an increasing sequence of eigenvalues $(\Lambda_k^0)_{k=1}^\infty$. On the other hand, for $\epsilon > 0$, one can use the Gelfand transform to render the original problem into another bounded domain problem depending on the additional parameter $\eta \in [0, 2\pi)$. For each fixed $\eta$ this problem again has a sequence of eigenvalues $(\Lambda_k^\epsilon(\eta))_{k=1}^\infty$. Moreover, by results of \cite{na17}, \cite{NaPl}, Theorem 3.4.6, and \cite{NaSpec}, Theorem 2.1,
the essential spectrum $\sigma$ of the   problem \eqref{problem}--\eqref{b-cond} 
equals 
\begin{equation}
\sigma = \bigcup_{k=1}^\infty \Upsilon_k^\epsilon \ , \ \ 
\Upsilon_k^\epsilon = \{  \Lambda_k^\epsilon (\eta) \, : \,
\eta \in [0,2 \pi) \} ,
\label{eq1.1}
\end{equation}
where the sets $\Upsilon_k^\epsilon$ are subintervals of the positive real axis, or bands of the spectrum. (For the use of this so called Bloch spectrum in other problems, see for example \cite{FG}, or \cite{AC}.) In general, those bands may overlap making $\sigma$ connected, but in   Theorem \ref{th3.1} we obtain asymptotic estimates  for the lower and upper endpoints of $\Upsilon_k^\epsilon$: we show
that   $\Lambda_k^0 \leq \Lambda_k^\epsilon(\eta) \leq \Lambda_k^0 + C_k \epsilon$
for all $k$ and $\eta$ and  for some constants $C_k >0$. In view of \eqref{eq1.1} this implies the existence of a spectral gap between $\Upsilon_k^\epsilon$ and $\Upsilon_{k+1}$ for small $\epsilon $ and  $k$ such that $\Lambda_k^0 \not= \Lambda_{k+1}^0$. However, since the estimates depend also on $k$, we can only open a gap for finitely many $k$, though the number of gaps tends to infinity as $\epsilon \to 0$.

The asymptotic position (as $\epsilon \to 0$) of the gaps is determined more accurately in Theorems \ref{th4.5} and  \ref{cor4.6}: those main results
state that 
\begin{equation*}
\Upsilon_k^\epsilon=(\Lambda_k^0+A_k\epsilon+O(\epsilon^{3/2}), \Lambda_k^0+B_k\epsilon+O(\epsilon^{3/2}))
\end{equation*}
where the numbers $A_k \leq  B_k$ depend linearly on the three dimensional
capacity of the set $\theta$. 
This  result also ensures that in case $A_k \not= B_k$ the bands $\Upsilon_k^\epsilon$ do not degenerate into single points, which means that the spectrum of the original problem indeed has a genuine band-gap structure. 
Facts concerning the numbers $A_k$, $B_k$ are discussed after Theorem \ref{cor4.6}.

As for the structure of this paper, we recall in Section \ref{sec:1.2} the  exact formulation of the linear water-wave problem, its variational formulation as well as the parameter dependent problem arising from the Gelfand transform, and the limit problem. Section \ref{sec:3} contains the formal asymptotic analysis which relates the spectral properties of the original problem with the limit problem and which is rigorously justified in Secion \ref{sec:4}. The main results, Theorems \ref{th3.1},
\ref{th4.5} and  \ref{cor4.6}  as well as Corollary \ref{cor3.2}  are also given in Section \ref{sec:4}.  The proofs are based on the  max-min principle and construction of suitable test functions adjusted to the geometric characteristics of the domains under study.

Acknowledgement. The authors want to thank Prof. Sergey A. Nazarov for many discussions on the topic of this work.

\subsection{Formulation of the problem, operator theoretic tools}
\label{sec:1.2}
Let us proceed with the exact formulation of the problem. 
We consider an infinite periodic channel \(\Pi^\epsilon\) (see \eqref{channel}), consisting of water containers connected by small apertures of diameter $O(\epsilon)$. The coordinates of the points in the channel are denoted by \(x=(x_1,x_2,x_3)=(y_1,y_2,z) = (y,z) \), and \(x'=(x_2,x_3)\) stands for the projection of $x$ to the plane \(\{x_1=0\}\). We choose the coordinate system in such a way that
the axis of the channel is in \(x_1\)-direction and the free surface is in the plane \(\{x_3=0\}\).

\begin{figure}
\begin{center}
 \includegraphics[width=10cm]{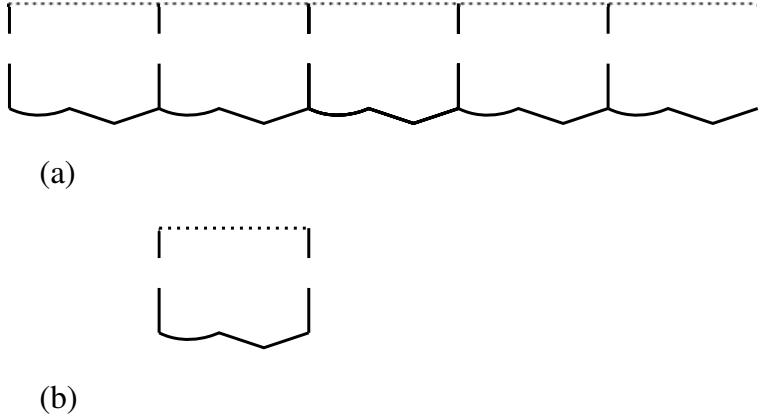}
\end{center}
\caption{Side view of the waveguide (a) and of the periodicity cell (b)}\label{fig1}
\end{figure}

\begin{definition} \rm
We describe the geometric assumptions on the periodicity cell in detail, as well as  some
related technical tools including the cut-off funcions. Let us  
denote by $\varpi_\bullet\subset {\mathbb R}^3$ a domain with a Lipschitz boundary and compact closure such that its intersections with \(\{x_1=0\}\)- and \(\{x_1=1\}\)-planes are simply connected planar
domains with positive area and  contain the points \({P^0}=(0,P_2,P_3)\) and \({P^1}=(1,P_2,P_3)\) 
with $P_3 < 0$, respectively; these points are fixed throughout the paper. 
Then the periodicity cell and its translates  are defined by setting (see Figure \ref{fig1})
\begin{equation}
\label{baths}
\varpi=\{x\in \varpi_\bullet: x_3<0, x_1\in (0,1)\},\quad
\varpi_j=\{x: (x_1-j,x_2,x_3)\in \varpi\}, \ j \in {\mathbb Z}.
\end{equation}
Furthermore, we assume that the set $\theta\subset {\mathbb R^2}$ is a bounded planar domain  containing the origin $(0,0)$ and that the boundary $\partial \theta$ is at least $C^2$-smooth.
We assume that $\theta$ is so small that the set $\{ 0 \} \times \overline{ \big( 2 \theta +(P_2, P_3) \big) }$ is contained in $\partial \varpi$ and $\sup \limits_{(x_2,x_3) \in \theta} (x_3 + P_3) =: d_\theta < 0$.  We define the apertures between the container walls as the sets
\begin{equation}
\label{apertures}
\theta_j^\epsilon=\big\{x=(j,x'):\, \epsilon^{-1} \big(x'-(P_2,P_3)\big)\in \theta\big\}, 
\ j \in {\mathbb Z}.
\end{equation}
It is plain that $x_3 < 0$ for $x \in \theta_j^\epsilon$ for all $0 < \epsilon \leq 1$,
by the choice of $d_\theta$. We shall need at several places a  cut-off function 
\begin{equation}
\label{1.11}
\chi_\theta \in C_0^\infty (\mathbb{R}^3)  ,
\end{equation} 
which is equal to one in a neighbourhood  of the set $\{ 0 \} \times \overline{\theta}$
and vanishes  outside another compact neighbourhood  of $\{ 0 \} \times \overline{\theta}$. More precisely, we require that  
\begin{eqnarray}
& & \big( {\rm supp}\, (\chi_\theta ) + (0,P_2,P_3) \big) \cap
\{ x_1 = 0\} \subset \partial \varpi , \nonumber \\ 
& &  
\big( {\rm supp}\, (\chi_\theta ) + (0,P_2,P_3) \big) 
\ \cap \{ x_1 >  0\} \subset  \varpi 
\label{1.11ab}
\end{eqnarray}
(this is possible by the specifications made on $\theta$)  and $\chi_\theta(x)$ vanishes, 
if $|x_1| \geq 1/4 $ or $ x_3 + P_3 \geq d_\theta/2$. We also assume that $\partial_{x_1} \chi_\theta= 0$, when $x_1 = 0 $.
Furthermore,  denoting $\chi_j (x) =  \chi_\theta (x- P^j )$, it follows from the above
specifications that $\chi_j(x) = 0$, if $x_3 \geq d_\theta/2$; in particular
$\chi_j$ vanishes on the free water surface $\gamma$. 
Finally, we shall need the scaled cut-off functions  
\begin{equation}
\label{defex}
X^\epsilon_j = \chi_\theta( \epsilon^{-1}(x-P^j)) .  
\end{equation}
It is plain that also $X^\epsilon_j$ vanishes on $\gamma$ for $0 < \epsilon \leq 1$
and that $X^\epsilon_j(x) = 1$ for $x \in \theta_j^\epsilon$, $j=0,1$. 
\end{definition}

\begin{definition} \rm
The periodic water channel is defined by 
\begin{equation}
\label{channel}
\Pi^\epsilon=\bigcup_{j\in{\mathbb Z}} ({\varpi_j\cup \theta_j^\epsilon}) ,
\end{equation}
and it will be the main object of our investigation. The free surface of the channel is denoted by
$\Gamma^{\epsilon}=\partial \Pi^{\epsilon}\cap \{x_3=0\}$, and the wall and bottom part of the boundary is $\Sigma^\epsilon=\partial \Pi^\epsilon\setminus \overline{\Gamma^\epsilon}$.
The boundary of the isolated container $\varpi$, the periodicity cell,
consists of the free surface $\gamma$ and the wall and bottom $\sigma^\epsilon$ 
with two apertures $\theta_0^\epsilon$ and $\theta_1^\epsilon$. 
\end{definition}

\begin{remark}
\label{rem1.1}
\rm
We shall  use the following general  notation. 
Given a domain $\Xi $, the symbol  $(\cdot,\cdot)_\Xi$ stands for the natural scalar product in $L^2(\Xi)$, and $H^k(\Xi)$, $k\in \mathbb{N}$,  for the standard Sobolev space of order $k$  on $\Xi$.  The norm of a function $f$ belonging to a Banach function space $X$ is denoted by $\Vert f ; X \Vert$. 
 For $r >0$ and $a \in \mathbb{R}^N$, $B_r (a)$ (respectively, $S_r(a)$ ) stands  for the Euclidean ball (resp. ball surface) with centre $a$ and radius $r$. By $C,c$ (respectively, $C_k$, $c_k$, $C(k)$ etc.) we mean positive constans (resp. constants depending on a parameter $k$) which do not depend on  functions or variables appearing in the inequalities, but which may still vary from place to place.  The gradient and Laplace
operators $\nabla$ and $\Delta$ act in variable $x$, unless otherwise indicated. 
\end{remark}

In the framework of the linear  water-wave theory 
we consider the spectral Steklov problem in the channel $\Pi^{\epsilon}$,
\begin{eqnarray}
-\Delta u^\epsilon(x) =0 &\quad &\mbox{for all} \ x\in \Pi^\epsilon, \label{problem} \\
\label{b-cond-1}
\partial_nu^{\epsilon}(x)=0 &\quad &\mbox{for a.e.} \ x\in\Sigma^\epsilon,\\
\label{b-cond} 
\partial_z u^\epsilon(x)=\lambda^\epsilon u^\epsilon(x) &\quad &\mbox{for a.e.} \
x\in\Gamma^\epsilon.
\end{eqnarray}
Here  \(u^\epsilon\) is the 
velocity potential,  $\lambda^\epsilon=g^{-1}\omega^2$ is a spectral parameter related  to the
frequency of harmonic oscillations $\omega>0$ and  the acceleration of gravity $g$. 
By the geometric assumptions made above, the outward normal derivative $\partial_n$ is defined
almost everywhere on $\Sigma^\epsilon$. It coincides with
$\partial_z=\partial/\partial_z$ on the free surface~$\Gamma^\epsilon$. 

The rest of this section is devoted to presenting the operator theoretic tools which will be needed later to prove our results: Gelfand transform, variational formulation of the boundary value problems, and max-min-formulas for eigenvalues.
The spectral problem (\ref{problem})--(\ref{b-cond}) can be transformed into a 
family of spectral problems in the periodicity cell using the Gelfand transform. We briefly recall its definition:
\begin{equation}
\label{Gelfand}
v(y,z)\mapsto\ V(y,z,\eta)=\frac{1}{\sqrt{2\pi}}\sum_{j\in \mathbb{Z}}
\exp(-i\eta(z+j))v(y,z+j),
\end{equation}
where $(y,z)\in \Pi^\epsilon$ on the left while $\eta\in[0,2\pi)$ and
$(y,z)\in\varpi$ on the right. As is well known, the Gelfand transform establishes an
isometric isomorphism between the Lebesgue spaces,
\begin{equation*}
L^2(\Pi^\epsilon)\simeq L^2(0,2\pi; L^2(\varpi)),
\end{equation*}
where $L^2(0,2\pi;B)$ is the Lebesgue space of functions with values in the Banach space $B$ endowed with the norm
\begin{equation*}
\|V;L^2(0,2\pi; B)\|=\left(\int_0^{2\pi}\|V(\eta);B\|^2d\eta\right)^{1/2}\,.
\end{equation*}
The Gelfand transform is also an isomorphism from the Sobolev space $H^l(\Pi^\epsilon)$ onto 
$L^2(0,2\pi; H^l_{\epsilon,\eta}(\varpi))$ for \(l=1,2\).
The space $H^2_{\epsilon,\eta}(\varpi)$ consists of  Sobolev functions $u$
which satisfy the quasi-periodicity conditions 
\begin{eqnarray}
\label{qua-1}
u(0,x')=e^{-i\eta}u(1,x'), &\quad& (0,x')\in \theta_0^\epsilon, \\
\label{qua-2}
\partial_{x_1}u(0,x')=e^{-i\eta}\partial_{x_1}u(1,x'), &\quad& (0,x')\in \theta_0^\epsilon,
\end{eqnarray}
whereas  $H^1_{\epsilon,\eta}(\varpi)$ is the Sobolev space with the condition
\eqref{qua-1} only.

Applying the Gelfand transform to the differential equation \eqref{problem}
and to the boundary conditions \eqref{b-cond-1}--\eqref{b-cond}, we obtain a family of model problems
in the periodicity cell $\varpi$ parametrized by the dual variable $\eta$,
\begin{eqnarray}
\label{model}
-\Delta U^\epsilon(x;\eta)=0, &\quad& x\in \varpi, \\
\label{model-b-cond-1}
\partial_nU^\epsilon(x;\eta)=0, &\quad& x\in \sigma^\epsilon,\\
\label{model-b-cond-2}
\partial_zU^\epsilon(x;\eta)=\Lambda^\epsilon (\eta) U^\epsilon(x; \eta), &\quad& x\in \gamma, \\
\label{quasiperiodic-1}
U^\epsilon(0,x';\eta)=e^{-i\eta}U^\epsilon (1,x';\eta), 
&\quad& x\in \theta_0^\epsilon,           \\
\label{quasiperiodic-2}
\partial_{x_1}U^\epsilon(0,x';\eta)=e^{-i\eta}\partial_{x_1}U^\epsilon (1,x';\eta),
&\quad& x\in \theta_0^\epsilon.
\end{eqnarray}
Here,  $\Lambda^\epsilon = \Lambda^\epsilon (\eta)$ is a new notation for the spectral
parameter $\lambda^\epsilon$. More details on the use of the Gelfand-transform can be found e.g. in 
\cite{NaRuTa}, Section 2.  

The apertures disappear at $\epsilon=0$ so in that case the also quasi-periodicity conditions cease to exist. Hence,  we can consider the problem \eqref{model}--\eqref{quasiperiodic-2} as a singular perturbation of the limit spectral problem 
\begin{eqnarray}
\label{model-0}
-\Delta U^0(x)=0, &\quad& x\in \varpi, \\
\label{model-0-b-cond-1}
\partial_nU^0(x)=0, &\quad& x\in \sigma,\\
\label{model-0-b-cond-2}
\partial_zU^0(x)=\Lambda^0 U^0(x), &\quad& x\in \gamma
\end{eqnarray}
with $\Lambda^0$ as a spectral parameter.

Our approach to the spectral properties of model and limit problems  is 
similar to  \cite{NaTa}, Sections 1.2, 1.3.. We first write 
the variational form of the problem (\ref{model})--(\ref{quasiperiodic-2}) 
for the unknown function $U^\epsilon \in H_{\epsilon,\eta}^1(\varpi)$
 as
\begin{equation}\label{spectralvar}
(\nabla U^\epsilon,\nabla
V)_{\varpi}=\Lambda^\epsilon
(U^\epsilon,V)_\gamma \, , \  V\in H^1_{\epsilon,\eta}(\varpi) ,
\end{equation}
and the corresponding variational formulation of the limit problem for 
\(U\in H^1(\varpi)\)  reads as
\begin{equation}\label{var_limit}
(\nabla U,\nabla V)_{\varpi}=\Lambda (U,V)_\gamma  \, , \  V\in H^1(\varpi) .
\end{equation}
We denote by ${\mathcal{H}}^\epsilon$ the space $H_{\epsilon,\eta}^1 (\varpi)$ endowed with the new scalar product
\begin{equation}
(u,v)_\epsilon = (\nabla u , \nabla v)_{\varpi} 
+ (u,v)_\gamma , \label{1.25}
\end{equation}
and define  a self-adjoint, positive and compact operator ${\mathcal{B}}^{\epsilon}(\eta):
{\mathcal{H}}^\epsilon \to {\mathcal{H}}^\epsilon$ using
\begin{equation}
({\mathcal{B}}^{\epsilon}(\eta) u,v)_\epsilon = (u,v)_\gamma .\label{1.26}
\end{equation}
The problem \eqref{spectralvar} is then equivalent to the 
standard spectral problem
\begin{equation}
{\mathcal{B}}^{\epsilon}( \eta) u = M^\epsilon u \label{1.28}
\end{equation}
with another spectral parameter 
\begin{equation}
M^\epsilon = (1 + \Lambda^\epsilon)^{-1}.\label{1.35}
\end{equation}
Clearly, the spectrum of ${\mathcal{B}}^{\epsilon}(\eta)$ consist of 0 and a decreasing sequence $(M_k^{\epsilon}(\eta))_{k=1}^\infty$ of eigenvalues, which moreover can be calculated from the usual min-max formula
\begin{equation}
M_k^\epsilon (\eta) = \min\limits_{E_k} \max\limits_{v \in E_k} \frac{({\mathcal{B}}^{\epsilon}(\eta) v,v)_\epsilon }{(v,v)_\epsilon} ,
\end{equation}
where the minimum is taken over all subspaces $E_k \subset {\mathcal{H}}^\epsilon$  of co-dimension $k-1$.
Using \eqref{1.25} and \eqref{1.26}, we can write a max-min formula for the eigenvalues of the problem \eqref{spectralvar}:
\begin{eqnarray}
\label{max-min-2}
& & \Lambda_k^\epsilon(\eta) = \frac1{M_k^\epsilon(\eta)} -1 
= \max\limits_{E_k} \min\limits_{v \in E_k}\frac{(\nabla v , \nabla v)_{\varpi} 
+ (v,v)_\gamma }{(v,v)_\gamma }-1  \nonumber \\
& = & \max\limits_{E_k} \min\limits_{v \in E_k} \frac{ \Vert \nabla v ; L^2(\varpi) \Vert^2 }{\Vert v ;L^2(\gamma) \Vert^2 } 
\end{eqnarray}
On the other hand, the connection \eqref{1.35} and the properties of the sequence
\hfill\break
$(M_k^{\epsilon}(\eta))_{k=1}^\infty$ mean that the eigenvalues \eqref{max-min-2}
form an unbounded sequence
\begin{equation}
\label{seq}
0 \leq \Lambda_1^\epsilon(\eta) \le \Lambda_2^\epsilon(\eta)  \le \ldots\le\Lambda_k^\epsilon(\eta)\le\ldots\to+\infty .
\end{equation}
The eigenfunctions can be assumed to form an orthonormal basis in the space \(L^2(\varpi)\).
The functions $\eta\mapsto\Lambda_k^\epsilon(\eta)$ are continuous and
$2\pi$-periodic (see for example  \cite{Kato}, Ch.\,9). Hence the sets
\begin{equation}\label{sets}
\Upsilon_k^\epsilon=\{\Lambda_k^\epsilon(\eta):\eta\in [0,2\pi)\}
\end{equation}
are closed connected segments, which may degenerate into single points; their relation to the original problem was already mentioned in \eqref{eq1.1}.

The spectral concepts of the limit problem 
\eqref{model-0}--\eqref{model-0-b-cond-2} can be treated in the same way as in
\eqref{1.25}--\eqref{seq}. Since the quasi-periodicity conditions vanish for $\varepsilon=0$,
the space ${\mathcal{H}}^\epsilon$ is replaced by $H^1(\varpi)$; the norm induced by \eqref{1.25}  is now equivalent to the  original Sobolev norm of $H^1(\varpi)$. We denote by $\mathcal{B} : H^1(\varpi) \to
H^1(\varpi)$ the operator defined as in \eqref{1.25}--\eqref{1.26}. 
The limit problem  has an  eigenvalue sequence  \( (\Lambda^0_k)_{k=1}^\infty \) like \eqref{seq}, however, neither the eigenvalues nor the operator $\mathcal{B}$ depend on  \(\eta\) (cf. \cite{NaRuTa},
Section 3). The first eigenvalue $\Lambda_1^{0}$ equals $0$, and the first eigenfunction is the  constant function. Analogously to \eqref{max-min-2} we can write
\begin{equation}
\label{max-min-1}
\Lambda^0_k=\max_{F_k}\min_{v\in F_k}
\frac{\|\nabla v ;L^2(\varpi)\|^2}{\|v;L^2(\gamma)\|^2},
\end{equation}
where again $F_k \subset H^1(\varpi)$ is running over all  subspaces of codimension $k-1$.
We denote by 
\begin{equation}
\label{1.45}
( U^0_k)_{k = 1}^\infty
\end{equation} 
an $L^2(\gamma)$-orthonormal sequence of eigenfunctions corresponding to the
eigenvalues \eqref{max-min-1}.

\begin{lemma}
\label{rem1.6}
For all $k$ there exists a  constant $C_k > 0$  such that
\begin{equation}
|U_k^0 (x)| \leq C_k \ , \ |\nabla U_k^0 (x)| \leq C_k 
\end{equation}
for all $x \in {\rm supp}\, (\chi_j)\cap \varpi$, $j=0,1$ (and hence for all $x \in {\rm supp}\, (X^\epsilon_j)\cap \varpi$, $0 < \epsilon \leq 1$).
\end{lemma}

\begin{proof} Let for example $j=0$ (the other case is treated similarly), and define the domains 
$G_1, G_2 \subset {\mathbb{R}^3} $ with $C^\infty$ boundary such that $\overline{G_0 } :=  {\rm supp}\, (\chi_j) 
\subset G_1 \subset \overline{ G_1  }  \subset G_2 \subset \{ x_3 < 0\}$ and $G_2$ still so
small that 
\begin{equation}
\label{1.45a}
G_2 \cap \{x_1 = 0\} \subset \partial \varpi \ \ {\rm and} \ \ 
\overline{G_2} \cap \{x_1 >  0\} \subset  \varpi .
\end{equation} 
As a consequence, these domains are smooth enough so that we can use the local 
elliptic estimates \cite{ADN}, Theorem 15.2, to the solutions  $U_k^0$ of the equation \eqref{model-0}: this yields
for every $l= 1,2,\ldots$, a constant $C_{l,k} >0$ such that 
\begin{equation}
\Vert U_k^0 ; H^{l+1}(G_n \cap \varpi) \Vert \leq C_{l,k} \big( \Vert  U_k^0 ; H^{l-1}(G_{n+1} \cap \varpi) \Vert
+ \Vert  U_k^0 ; L^2 (G_{n+1} \cap \varpi) \Vert \big)  \nonumber
\end{equation}
for $n=0,1$. Applying this first with $n=1$ and $l=1$ we get a bound
for $\Vert U_k^0 ; H^{2}(G_1 \cap \varpi) \Vert $ and then, with $n=0$ and $l=2$, for $\Vert U_k^0 ; H^{3}(G_0 \cap \varpi) \Vert $. The standard embeddings
$H^2( G_1 \cap \varpi) \subset C_B (G_0 \cap \varpi)$ and  $H^3(G_0\cap \varpi) \subset C_B^1 (G_0\cap \varpi)$
imply the result. \ \ $\Box$
\end{proof}

\section{The formal asymptotic procedure}
\label{sec:3}
\subsection{The case of a simple eigenvalue}
\label{sec:3.1}
 
To describe the asymptotic behaviour (as $\epsilon \to 0$) of the eigenvalues  \(\Lambda_k^\epsilon(\eta)\)
of the problem \eqref{model}-\eqref{quasiperiodic-2} we consider first  the case $\Lambda_k^0$ is a simple eigenvalue of the problem 
\eqref{model-0}-\eqref{model-0-b-cond-2} for some fixed $k$.
Let us  make the following \textit{ansatz}:
\begin{equation}\label{ansatz1}
\Lambda_k^\epsilon(\eta)=\Lambda_k^0+\epsilon\Lambda_k'(\eta)+
\widetilde{\Lambda}_k^\epsilon(\eta),
\end{equation}
where 
\(\Lambda_k'(\eta)\) is a correction term and 
\(\widetilde{\Lambda}_k^\epsilon(\eta)\) a small remainder to be evaluated and estimated. In this 
section we derive  the expression \eqref{2.20} for $\Lambda_k'(\eta)$, cf. also \eqref{L-mult-1} and \eqref{L-mult-2}, and the 
remainder will be  treated in Section \ref{sec:4.2}

The corresponding asymptotic ansatz for the eigenfunction reads as follows:
\begin{eqnarray}
\label{ansatz2}
U_k^{\epsilon}(x;\eta)&=&U_k^0(x)\\
&+& \chi_0(x)w_{k0}(\epsilon^{-1}(x-{P^0}))+ \chi_1(x)w_{k1}(\epsilon^{-1}(x-{P^1}))\nonumber\\
&+& \epsilon U_k'(x;\eta)+\widetilde{U}_k^\epsilon(x;\eta),\nonumber
\end{eqnarray}
where \( ( U^0_k)_{k = 1}^\infty\) is as in \eqref{1.45}.
The functions \(w_{k0}\) and \(w_{k1}\) are of boundary layer type, and \(\chi_j\)
is given above \eqref{defex}.

The boundary layers \(w_{kj}\) depend on the ``fast'' variables 
(``stretched'' coordinates)
\begin{equation*}
\xi^{j}=(\xi^j_1,\xi^j_2,\xi^j_3)=\epsilon^{-1}(x-{P^j}),\quad j=0,1.
\end{equation*} 
They are needed to compensate the fact that the leading term \(U^0_k\) 
in the expansion \eqref{ansatz2} does not satisfy the quasi-periodicity conditions
\eqref{quasiperiodic-1}--\eqref{quasiperiodic-2}.
By Lemma \ref{rem1.6} and the mean value theorem, the eigenfunction \(U_k^0(x)\) has the representation 
\begin{equation*}
U_k^0(x)=U_k^0({P^j})+O(\epsilon), \quad x\in \theta_j^\epsilon
\end{equation*}
near the points \({P^j},\,j=0,1\). We look for 
 \(w_{k0}\) and \(w_{k1}\) as  the solutions of the
problems 
\begin{eqnarray*}
\Delta_{\xi^0} w_{k0}(\xi^{0})=0, &\quad& \xi^0_1>0,\\
\partial_{\xi^0_1} w_{k0} (\xi^0)=0, &\quad& 
\xi^0\in \{0\}\times ({\mathbb R}^2\setminus \overline{\theta}), \\
w_{k0}(\xi^0)=a_{k0}, &\quad& \xi^0\in\{0\}\times \theta,
\end{eqnarray*}
and
\begin{eqnarray*}
\Delta_{\xi^1} w_{k1}(\xi^1)=0, &\quad& \xi^1_1<0,\\
\partial_{\xi^1_1} w_{k1} (\xi^1)=0, &\quad& 
\xi^1\in \{0\}\times ({\mathbb R}^2\setminus \overline{\theta}), \\
w_{k1}(\xi^1)=a_{k1}, &\quad& \xi^1\in\{0\}\times \theta\,
\end{eqnarray*}
in the half spaces \(\{\xi^0_1>0\}\) and \(\{\xi^1_1<0\}\), respectively;
 the meaning of  the numbers $a_{kj}$ will be explained below. 
Both of the functions \(w_{kj},\,j=0,1,\) can be extended to even harmonic functions 
in the exterior of the set \(\{0\}\times\theta\):
\begin{eqnarray}
\label{harmonic}
\Delta_{\xi^j} w_{kj}(\xi^j)=0, &\quad& \xi^j\in {\mathbb R}^3\setminus (\{0\}\times\overline{\theta}),\\
w_{kj}(\xi^j)=a_{kj}, &\quad& \xi^j\in \partial (\{0\}\times\overline{\theta}) \,.\nonumber
\end{eqnarray}
Furthermore, the problem (\ref{harmonic}) admits a solution (see \cite{PolyaSzego}) 
\begin{eqnarray}
\label{layer0}
& & w_{kj}(\xi^j)=a_{kj}\frac{\capas_3 \theta}{|\xi^j|}+\widetilde{w}_{kj}(\xi^j), \\
\label{new2}
& & \widetilde{w}_{kj}(\xi^j)= O(|\xi^j|^{-2})\ , 
\ \nabla_{\xi^j} \widetilde{w}_{kj}(\xi^j)= O(|\xi^j|^{-3}),
\end{eqnarray} 
where \(\capas_3 (\theta)\) is the 3-dimensional capacity of the set \(\{0\}\times\theta\) and \eqref{new2} concerns large $\xi^j$-behaviour.
Moreover, the solution has a finite Dirichlet integral: 
\begin{eqnarray}
& & \int_{{\mathbb R}^3} \big| \nabla_{\xi^j } w_{kj} (\xi^j) \big|^2 d\xi^j 
\leq C 
\label{new1} 
\end{eqnarray} 
for some constant $C>0$. 

We aim to choose the coefficients \(a_{kj}\)  such that $U_k^\epsilon$ 
satisfies the quasi-periodicity conditions 
\eqref{quasiperiodic-1}--\eqref{quasiperiodic-2} . Clearly, for each $\epsilon>0$
\begin{eqnarray*}
U_k^\epsilon({P^0}; \eta)&=&e^{-i\eta}U_k^\epsilon({P^1}; \eta),\\
\partial_{x_1}U_k^\epsilon({P^0};\eta)&=&e^{-i\eta}\partial_{x_1}U_k^\epsilon({P^1};\eta),
\end{eqnarray*}
which together with the asymptotic expansion \eqref{ansatz2} yield the relations
\begin{equation*}
U_k^0({P^0})+a_{k0}=e^{-i\eta}(U_k^0({P^1})+a_{k1}) \mbox{ and } a_{k0}=-e^{-i\eta}a_{k1}
\end{equation*}
for the coefficients. Hence, 
\begin{equation}
\label{ab}
a_{k1}=-e^{i\eta} a_{k0}, \quad
a_{k0}=\frac{1}{2}\left(e^{-i\eta}U_k^0({P^1})-U_k^0({P^0})\right).
\end{equation}

Now we can write a model problem for the main asymptotic correction term \(U_k'\):
\begin{eqnarray}
\label{U'-eq}
-\Delta U_k' (x;\eta)= \Delta 
W_k(x) 
&\quad& x\in \varpi, \\
\label{U'-b-cond-1}
(\partial_z-\Lambda_k^0)U_k'(x;\eta)=\Lambda_k'(\eta)U_k^0(x), &\quad& x\in\gamma,\\
\label{U'-b-cond-2}
\partial_n U_k'(x ; \eta)=0, &\quad& x\in \sigma\, ,
\end{eqnarray}
where we denote 
\begin{equation}
W_k(x) = \left(\sum_{j=0}^1  \chi_j(x)\frac{a_{kj}\capas_3(\theta)}{|x-{P^j}|}\right) 
, \quad x\in \varpi,. \label{2.19ce}
\end{equation}
In addition to $U_k'$, the problem \eqref{U'-eq}--\eqref{U'-b-cond-2}
will also determine the number  \(\Lambda_k'(\eta)\) in a unique way for every $k$ and $\eta$.
This will follow by requiring the  solvability condition to hold in the Fredholm alternative, see Lemma \ref{lem2.1} and its proof,  below. 
Indeed, using the Green formula
and the  normalization in \eqref{1.45} we write ($ds $ is the surface measure): 
\begin{eqnarray}
&&\Lambda_k'(\eta)=\Lambda_k'(\eta)\|U_k^0;L^2(\gamma)\|^2 = 
\nonumber \\
&=&\int_\gamma\left( \partial_z U_k'(x ;\eta)- 
\Lambda_k^0 U_k'(x;\eta)\right)\overline{U_k^0(x)}\,ds(x)= 
\nonumber  \\
&=&\int_{\partial\varpi} 
\left( \overline{U_k^0(x)}
\partial_n U_k'(x;\eta) -
U_k'(x;\eta)\overline{\partial_n U_k^0(x) }
\right)\,ds(x)= 
\nonumber \\
&=&\int_\varpi \overline{U_k^0(x)} \Delta U_k'(x;\eta)=
- \int_\varpi \overline{U_k^0(x)} \Delta W_k(x)
\,dx\,.  
\label{2.19cd}
\end{eqnarray}
Taking into account that the last integral converges absolutely and using the Green formula again yield
\begin{eqnarray*}
\Lambda_k'(\eta)
&=&\lim_{r\to 0} 
\sum_{j=0}^1\int_{S_r({P^j})\cap\varpi}\overline{U_k^0(x)}
\partial_n\left(-\frac{a_{kj}\capas_3(\theta)}{|x-{P^j}|}\right)\,ds(x)\\
&=& -2\pi\capas_3(\theta)\left(a_{k0}\overline{U_k^0({P^0})}+
a_{k1}\overline{U_k^0({P^1})}\right);
\end{eqnarray*}
see \eqref{2.19ce} and Remark \ref{rem1.1} for notation. 
According to \eqref{ab} we finally obtain
\begin{equation}
\Lambda'_k(\eta)=
\pi\capas_3(\theta) |U_k^0({P^0})-e^{-i\eta}U_k^0({P^1})|^2\,.
\label{2.20}
\end{equation}

\begin{lemma}
\label{lem2.1}
Choosing $\Lambda_k'(\eta)$ as in \eqref{2.20},
the problem \eqref{U'-eq}--\eqref{U'-b-cond-2} has a 
solution $U_k' \in H^1(\varpi)$.
\end{lemma}

\begin{proof}
The variational formulation of the problem \eqref{U'-eq}--\eqref{U'-b-cond-2} reads as
\begin{equation}
\label{2.20ac}
(\nabla U_k' , \nabla V)_\varpi - \Lambda_k^0(U_k', V)_\gamma
= (\nabla W_k , \nabla V)_\varpi - \Lambda_k' (U_k^0, V)_\gamma .
\end{equation}
We remark that the function $1/|x-P^j|$ is harmonic in $\varpi$, and since
$\chi_j$ equals  constant one in a neighbourhood of $P^j$, the function $\Delta W_k$ vanishes there, hence, $W_k$  and $\nabla W_k$ are  smooth as well as uniformly bounded everywhere 
in $\varpi$. Moreover, $U_k^0 \in L^2 (\gamma)$.

Using the definition  of the operator $\mathcal{B}: H^1(\varpi) \to 
H^1(\varpi) $ (cf. \eqref{1.25}, \eqref{1.26} and the remarks above \eqref{max-min-1}) 
we can rewrite \eqref{2.20ac} as follows: 
\begin{equation}
(U_k',V)_{0}-(\Lambda^0_k+1)(\mathcal{B}U_k',V)_0
=(W_k,V)_0-\Lambda_k'(\mathcal{B}U_k^0,V)_0
-(\mathcal{B}W_k,V)_0
\end{equation} 
which means that $U'_k$ must be a solution of the equation
\begin{equation}
(\mathcal{B}-M_k^0)U'_k=-M_k^0(W_k-\mathcal{B}W_k-\Lambda'_k\mathcal{B}U^0_k).
\label{2.20xx}
\end{equation}
Notice that $U_k^0$ is the solution of the homogeneous problem \eqref{2.20xx}, so,
by the Fredholm alternative, \eqref{2.20xx} is solvable, if and only if the right hand 
side of it is orthogonal  to the function $U_k^0$. 
This condition is 
satisfied by choosing $\Lambda_k'(\eta)$ as above, since
$$
(W_k-\mathcal{B}W_k-\Lambda'_k\mathcal{B}U^0_k,U^0_k)_0=
(\nabla W_k,\nabla U_k^0)_{\varpi}-\Lambda'_k\|U^0_k;L^2(\gamma)\|^2=0\, ,
$$ 
by \eqref{2.19cd} and $(\nabla W_k,\nabla U_k^0)_{\varpi}
= - (\Delta  W_k,  U_k^0)_{\varpi}$. This last identity follows from the first
Green formula, because the normal derivative of $W_k$ vanishes on $\partial \varpi$
due to the properties of the function $\chi_j$, see below \eqref{1.11ab}.  \ \ $\Box$
\end{proof}

\subsection{The case of a multiple eigenvalue}
\label{sec:3.2}
In this section we complete the asymptotic analysis by studying the behaviour of eigenvalues $\Lambda_k^\epsilon(\eta)$ in the case some $\Lambda_k^0$ has  multiplicity $m $ greater than one: we have
$$
\Lambda_{k-1}^0 < \Lambda_{k}^0 = \ldots  =  \Lambda_{k + m-1}^0 < \Lambda_{k+m}^0 .
$$ 
The ansatz \eqref{ansatz1} is used again. Furthermore, as in  \eqref{1.45} we denote by $(U_{k+j}^0)_{0\leq j\leq m-1} \subset L^2(\gamma)$
an orthonormal system of eigenfunctions associated with the eigenvalue $\Lambda_k^0$. 
Any eigenfunction $U^0$ corresponding to $\Lambda^0_k$ can be presented
as a linear combination
\begin{equation*}
U^0(x)=\sum_{j=0}^{m-1}\alpha_j U_{k+j}^0(x).
\end{equation*}
Analogously to \eqref{ansatz2} we introduce the asymptotic ansatz
\begin{eqnarray}
\label{ansatz20}
U^{\epsilon}(x;\eta)&=&U^0(x)\\
&+& \chi_0(x)w_{k0}(\epsilon^{-1}(x-{P^0}))+ \chi_1(x)w_{k1}(\epsilon^{-1}(x - P^1))\nonumber\\
&+& \epsilon U'(x;\eta)+\widetilde{U}^\epsilon(x;\eta).\nonumber
\end{eqnarray}
Using the same argumentation as in the previous section we construct the boundary layers $w_{kj}$, $j=0,1$, which satisfy the conditions
\begin{equation*}
\label{layer00}
w_{kj}(\xi^j)=a_{kj}\frac{\capas_3 \theta}{|\xi^j|}+O(|\xi^j|^{-2}) ; 
\end{equation*}
here the coefficients \(a_{kj}\) come   from the equations \eqref{ab}, where $U_k^0$ is replaced by $U^0$.
The main asymptotic term $U'$ is also treated in the same way as in Section \ref{sec:3.1}.  
To use the Fredholm alternative for finding \(\Lambda_{k+j}'(\eta)\), $j=0, \ldots , m-1$, we write 
\begin{equation*}
\Lambda_{k+j}'(\eta)\alpha_j = \Lambda_{k+j}'(\eta)(U^0,U_{k+j}^0)_{\gamma} ,
\end{equation*}
and making use of the Green formula as above we get
\begin{equation*}
\Lambda_{k+j}'(\eta)\alpha_j =\sum_{l=0}^{m-1} \beta_{lj}\alpha_j,
\end{equation*}
where 
\begin{equation*}
\beta_{lj}=\pi\capas_3(\theta)(U_{k+l}^0({P^0})-e^{-i\eta}U_{k+l}^0({P^1}))\overline{(U_{k+j}^0({P^0})-e^{-i\eta}
U_{k+j}^0({P^1}))}.
\end{equation*}
Hence, $\Lambda_{k+j}'(\eta)$ is an eigenvalue of the matrix $B=(\beta_{lj})_{l,j=0}^{m-1}$. This matrix has rank one, because it can be represented in the form $B=\overline{v} v^\top $, where $v$ is a vector with components $v_{j+1}=U_{k+j}^0({P^0})-e^{-i\eta}U_{k+j}^0({P^1})$,
$j= 0, \ldots , m-1$. This means that 
\begin{eqnarray}
\label{L-mult-1}
\Lambda'_k(\eta)&=&\pi\capas_3(\theta)\sum_{l=0}^{m-1} |U_{k+l}^0({P^0})-e^{-i\eta}U_{k+l}^0({P^1})|^2,\\
\label{L-mult-2}
\Lambda'_{k+j}(\eta)&=&0, \quad 1\leq j\leq m -1\,.
\end{eqnarray}

\section{Existence and position of spectral gaps}
\label{sec:4}
\subsection{Existence of gaps}
\label{sec:4.1}

The first estimate on the  eigenvalues of the problem \eqref{spectralvar} can
now be stated as follows.

\begin{theorem}
\label{th3.1}
For any \(k\in\mathbb{N}\) there are numbers \(\epsilon_k>0\) and  \(C_k>0\) such that for every
\(\epsilon\in(0,\epsilon_k) \) and any dual variable \(\eta\in [0,2\pi)\), the eigenvalues of the
problem \eqref{spectralvar} and the eigenvalues of the limit problem \eqref{var_limit} are related
as follows:
\begin{equation}\label{eigenbound}
\Lambda^0_k\leq\Lambda^\epsilon_k(\eta)\leq\Lambda^0_k+C_k\epsilon.
\end{equation}
\end{theorem}

As mentioned  in the introduction (see the explanations around \eqref{eq1.1} and \eqref{sets}), this result implies the existence of any prescribed number of  gaps in the essential spectrum $\sigma$, since \eqref{eigenbound} also establishes an estimate for the endpoints of the intervals $\Upsilon_k^\epsilon$. To prove that result one needs to take enough many distinct eigenvalues $\Lambda_k^0$ and a small enough $\epsilon$.

\begin{corollary}
\label{cor3.2}
Given any number $N \in \mathbb{N}$, the essential spectrum $\sigma$ of the problem \eqref{problem}--\eqref{b-cond} on $\Pi^\epsilon$ has at least $N$ gaps, if $\epsilon $ is small enough.
\end{corollary}

{\it Proof of Theorem \ref{th3.1}.} We apply the  max-min-principle described in Section \ref{sec:1.2}
and  first prove the estimate
\begin{equation}
\label{max-min}
\Lambda_k^\epsilon (\eta)  \geq \Lambda_k^0 .
\end{equation}
Indeed, we recall that in the equations \eqref{max-min-1}
and \eqref{max-min-2} both 
$F_k \subset H^1(\varpi)$ and $E_k \subset {\mathcal{H}}^\epsilon = H^1_{\epsilon,\eta}(\varpi)$ are 
arbitrary  subspaces of co-dimension $k-1$.
Since $H^1_{\epsilon,\eta}(\varpi)\subset H^1(\varpi)$, each $E_k$ is contained in some  $F_k$, and thus the infimum in \eqref{max-min-1} is smaller
than that in \eqref{max-min-2}.

So we turn to the upper estimate in \eqref{eigenbound} and fix a $k \in \mathbb{N} $.
Let the eigenfunctions $U^0_j$ be as in \eqref{1.45} 
and let $H_k\subset{\mathcal{H}}^{\epsilon}(\varpi)$ be a subspace
spanned by the functions $Y^\epsilon U_j^0$, where $j=1, \ldots,k$ and  
\begin{eqnarray}
Y^\epsilon = 1 - X^\epsilon_0 - X^\epsilon_1 \in C^\infty(\varpi) \label{3.111}
\end{eqnarray}
and $X^\epsilon_j$ are as in \eqref{defex}.
We remark that the functions $Y^\epsilon U_j^0$ satisfy the  quasi-periodicity condition \eqref{qua-1} in the definition of the space ${\mathcal{H}}^\epsilon$,
since $Y^\epsilon$ vanishes in a neighbourhood of the apertures, see the remarks around \eqref{defex}. Moreover, the sequence  
$\big( Y^\epsilon U_1^0$, $Y^\epsilon U_2^0$, \ldots, $Y^\epsilon U_k^0 \big)$ is  still  linearly independent, due to the $L^2(\gamma)$-orthogonality in \eqref{1.45}
and the fact that $Y^\epsilon$ equals 1 in the set $\gamma$. Hence, the dimension
of $H_k$ is $k$.

If $E_k$ is an arbitrary subspace of  ${\mathcal{H}}^\epsilon$ of co-dimension $k-1$
(cf.\,\eqref{max-min-2}), the intersection $E_k \cap H_k$ contains a non-trivial linear combination
\begin{equation}
U(x)=Y^\epsilon(x)\sum_{j=1}^k a_j U^0_j(x), \quad \sum_{j=1}^k |a_j|^2=1 .
\label{3.35}
\end{equation}
By the remarks just above we have
$
\Vert U ; L^2(\gamma)\Vert = 1 .
$ 
Hence, from \eqref{max-min-2} and \eqref{1.45} we infer that
\begin{eqnarray}
\Lambda^\epsilon_k(\eta)&\leq& \frac{\left\|\nabla U;L^2(\varpi)\right\|^2}{\|U;L^2(\gamma)\|^2} 
=\|\nabla U;L^2(\varpi)\|^2 \nonumber \\ 
&=& \Big\|\nabla \Big( (Y^\epsilon - 1)\sum_{j=1}^ka_j  U_j^0 \Big)
+ \nabla \Big( \sum_{j=1}^k a_j U_j^0\Big)  ;L^2(\varpi)\Big\|^2
\nonumber \\ 
& = & \Big\| \nabla \Big( \sum_{j=1}^k a_j U_j^0\Big) ; L^2(\varpi)\Big\|^2+
2\Big( \nabla \Big( (-X^\epsilon_0 -X^\epsilon_1)\sum_{j=1}^ka_j U_j^0 \Big) ,
\nabla \Big( \sum_{j=1}^k a_j U_j^0\Big) \Big)_\varpi \nonumber \\
&+&
\Big\|\nabla \Big( (X^\epsilon_0 + X^\epsilon_1)\sum_{j=1}^ka_j  U_j^0 \Big) ;L^2(\varpi)\Big\|^2
\label{3.38}
\end{eqnarray}
To evaluate the first term on the right hand side notice that the functions $U_j^0$ satisfy
\eqref{var_limit} so that the $L^2(\gamma)$-orthogonality of \eqref{1.45} implies
\begin{eqnarray}
\label{3.40}
& & \Big\|\nabla \Big( \sum_{j=1}^k a_j  U_j^0 \Big); L^2(\varpi)\Big\|^2
= \sum_{j,l=1}^k a_j a_l (\nabla U_j^0 ,\nabla U_l^0)_\varpi \nonumber \\
& = & \sum_{j,l=1}^k a_j a_l \Lambda_j^0 ( U_j^0 , U_l^0)_\gamma
= \sum_{j=1}^k a_j^2 \Lambda_j^0  \leq \Lambda_k^0  , 
\end{eqnarray}
where the last inequality follows from \eqref{3.35} and the fact that
the eigenvalues $\Lambda_j^0$ are indexed in increasing order.
Furthermore, we use Lemma \ref{rem1.6} as well as the facts that the supports of $X^\epsilon_l$, $l=0,1$,
have measure of order $\epsilon^3$,  $|\nabla X^\epsilon_l|$ are of order $\epsilon^{-1}$, and $| a_j| \leq 1$ to estimate
\begin{eqnarray}
& & \Big\vert \Big( \nabla \Big( (X^\epsilon_0 + X^\epsilon_1)\sum_{j=1}^ka_j U_j^0 \Big) ,
\nabla \Big( \sum_{j=1}^k a_j U_j^0\Big) \Big)_\varpi
\Big\vert  \nonumber \\
& \leq & k^2  
\Big( \sup\limits_{ x \in  S }  \big(1, |U_j^0(x)|, |\nabla U_j^0(x) |\big)\Big)^2 \ \sup\limits_{ x \in  S } \big(1, | \nabla X^\epsilon_l(x)| \big) 
 \int_{ S }  dx 
\leq C_k \epsilon^{2} ,  \nonumber \\
& &  \Big\| \nabla \Big( (X^\epsilon_0+X^\epsilon_1)\sum_{j=1}^ka_j  U_j^0 \Big)
;L^2(\varpi)\Big\|^2 
\nonumber \\
\label{3.39}
& \leq & k^2  \Big( 
\sup\limits_{ x \in  S }  \big(1, |U_j^0(x)|, |\nabla U_j^0(x) |\big) \Big)^2
\Big(  \sup\limits_{ x \in S } \big(1, | \nabla X^\epsilon_l(x)| \big) 
\Big)^2  \int_{ S }  dx 
\leq C_k \epsilon ,
\end{eqnarray}
where $S = {\rm supp} (X^\epsilon_0+X^\epsilon_1)$. 
Combining this with \eqref{3.40} and \eqref{3.38} yields the result. 
\ \ $\Box$

\subsection{Asymptotic position of spectral bands}
\label{sec:4.2}

In this section we shall prove the validity of the asymptotic ansatz \eqref{ansatz1}, see
Theorem \ref{th4.5}. This yields our main result concerning the asymptotic position
of the spectral bands, Theorem \ref{cor4.6}.

We start the proof by recalling   a classical lemma on near eigenvalues and eigenvectors (see \cite{ViLu} and also, e.g., \cite{BiSo}).

\begin{lemma}\label{nearEig}
Let ${\mathcal{T}}$ be a selfadjoint, positive, and compact operator in a Hilbert
space ${\mathcal{H}}$. 
If a number $\mu>0$ and an element ${\mathcal{V}}\in {\mathcal{H}}$ satisfy
$\|{\mathcal{V}} ; {\mathcal{H}}\|=1$ and $\|{\mathcal{T}\mathcal{V}}-\mu {\mathcal{V}}; {\mathcal{H}}
\|=\tau\in(0,\mu)$, then the segment $[\mu -\tau,\mu+\tau]$ contains at least one
eigenvalue of ${\mathcal{T}}$. 
\end{lemma}

To apply Lemma \ref{nearEig} to the operator ${\mathcal{B}}^\epsilon(\eta)$
of \eqref{1.26}, we fix an arbitrary  $k$ and, keeping in mind the
formula \eqref{1.35},  define the approximate $k$:th eigenvalue and
eigenvector of ${\mathcal{B}}^\epsilon(\eta)$ by 
\begin{eqnarray}
\label{use lemma}
\mu_k &=&
(1+\Lambda_k^0+\epsilon\Lambda'_k(\eta))^{-1}, \\
{\mathcal{V}}_k(x)&=&\|{\mathcal{U}}_k; \cH^\epsilon \|^{-1} {\mathcal{U}}_k(x), \nonumber
\end{eqnarray}
where
\begin{eqnarray}
\label{calU}
&&{\mathcal{U}}_k(x)= (1-X^\epsilon_0(x)-X^\epsilon_1(x)) U_k^0(x)\\
&+&X^\epsilon_0(x) U_k^0({P^0}) + X^\epsilon_1(x) U_k^0({P^1}) \nonumber\\
&+&\chi_0 (x)  w_{k0}(\epsilon^{-1}(x-{P^0})) + 
\chi_1(x) w_{k1}(\epsilon^{-1}(x-{P^1})) \nonumber\\
&+& \epsilon (1-X^\epsilon_0(x)-X^\epsilon_1(x)) U'_k(x,\eta) , \nonumber
\end{eqnarray}
$U_k^0$ is as in \eqref{1.45}, $X^\epsilon_j(x)=\chi_\theta(\epsilon^{-1}(x-{P^j}))$
and $\chi_j$ are as in  \eqref{defex}.

We need a lower bound for the norm of $\cU_k$.

\begin{lemma}
For all $k$ there exists a constant $C_k>0$ such that
\begin{equation}                                                      
\label{lem1}
\Bigl|\|{\mathcal{U}}_k; \cH^{\epsilon}\|^2 -1 - \Lambda^0_k \Bigr|\leq C_k \epsilon^{1/2}\,.
\end{equation}
\end{lemma}

\begin{proof}
Recall that the expression for $\cU_k$, \eqref{calU}, contains the term
$U_k^0$; let us denote $\widetilde \cU_k := \cU_k - U_k^0$. By \eqref{1.25}, \eqref{max-min-1}--\eqref{1.45}, \eqref{var_limit}, we have $\Vert U_k^0 ; \cH^\epsilon \Vert^2 = \|U^0_k;L^2(\gamma)\|^2 + \|\nabla U^0_k; L^2(\varpi)\|^2
= 1 + \Lambda^0_k$. Hence,  by the Cauchy-Schwartz inequality,
\begin{eqnarray}
\label{3.85}
& & \Big| \|{\mathcal{U}}_k; \cH^{\epsilon}\|^2 -1 - \Lambda^0_k \Big|  = 
\Big|  2 (U_k^0 , \widetilde \cU_k )_\epsilon + \Vert \widetilde \cU_k; \cH^\epsilon \Vert^2 \Big|  
\nonumber \\
& \leq & 
2\sqrt{1 + \Lambda_k^0} \Vert  \widetilde \cU_k; \cH^\epsilon \Vert +
\Vert  \widetilde \cU_k; \cH^\epsilon \Vert^2 .
\end{eqnarray}

Taking into account the definition of  the norm of $\cH^\epsilon$, the formula \eqref{calU} and the fact that the functions $\chi_j$ and 
$X^\epsilon_{j}$ vanish on 
$\gamma$ we find that  $\Vert  \widetilde \cU_k; \cH^\epsilon \Vert$ is bounded
by the sum  of the expressions
\begin{eqnarray}
\label{first}
&&\big\|\nabla \big( X^\epsilon_j(U^0_k-U^0_k({P^j}))\big);L^2(\varpi)
\big\| \, , \ j = 0,1, 
\\
\label{second}
&&\big\|\nabla\big( \chi_j w_{kj} (\epsilon^{-1}(x-{P^j}))\big);L^2(\varpi)\big\| \,  \ j=0,1, 
\\
\label{third}
&&\|\epsilon (1-X^\epsilon_0-X^\epsilon_1)U_k'; H^1(\varpi)\|.
\end{eqnarray}

First we use the observation
that the supports of the functions $ X^\epsilon_{j}$ and $\nabla X^\epsilon_{j}$ 
are contained in  balls of radius $O(\epsilon)$ and that $|U_k^0(x)|$ and $|\nabla U_k^0 (x)|$ are uniformly bounded in these balls (Lemma \ref{rem1.6}), hence $U^0_k(x)-U^0_k({P^j})=O(\epsilon)$ there. So, \eqref{first} can be bounded by a constant times
\begin{eqnarray}
\label{3.46}
& & \int_\varpi |U_k^0  -U_k^0 (P_j)|^2  \, |\nabla X^\epsilon_j|^2  dx
+ \int_\varpi |\nabla  U_k^0 |^2 \,  |X^\epsilon_j|^2  dx \nonumber  \\
&\leq &
C \Big( \int_{{\rm supp} X^\epsilon_j } \epsilon^2 \epsilon^{-2} dx 
+ \int_{{\rm supp} X^\epsilon_j }dx \Big)^{1/2} \leq C \epsilon^{3/2}.
\end{eqnarray}

We estimate the terms \eqref{second} using the fact that
the support of the function $\nabla \chi_j$ is contained
in a set $\{ c \leq  |x - P^j| \leq C\} =: {\mathcal S}_j$ for some constants $0 < c < C$ (see above \eqref{defex}), hence, by the estimate \eqref{layer0}--\eqref{new2},
\begin{equation}
\label{3.03}
\big| {w}_{kj}(\epsilon^{-1}(x-P^j))  \big| 
\leq C \epsilon |x - P^j|^{-1} 
\ \ {\rm for}  \ x \in {\mathcal S}_j .
\end{equation}
Applying \eqref{new1} yields
\begin{eqnarray*}
&&\int_\varpi \big|\nabla \big(\chi_j(x)  w_{kj}(\epsilon^{-1}(x-{P^j}))
\big)\big|^2dx  \\
&\leq &
\int_{{\mathcal S}_j}  |w_{kj}(\epsilon^{-1}(x-{P^j}))|^2dx
+ \int_{\varpi} \big| \nabla \big(w_{kj} (\epsilon^{-1}(x-{P^j}))\big)\big|^2dx
\\
&\leq &\int_{{\mathcal S}_j} C \epsilon^2 |x - P^j|^{-2}  dx   + \epsilon^{3}
\epsilon^{-2} \int_{{\mathbb R}^3 }  \big| \nabla_{\xi^j} w_{kj} (\xi^j) \big|^2 
  d \xi^j
  \\
&\leq & C_1 \epsilon .
\end{eqnarray*}
Finally, by Lemma \ref{lem2.1}, \(U_k'\) 
belongs to the space  
\(H^1(\varpi)\). For the terms \eqref{third} we thus get the
bound 
\begin{eqnarray}
& & \|\epsilon (1-X^\epsilon_0-X^\epsilon_1)U_k'; H^1(\varpi)\| \nonumber \\
& \leq &   C \epsilon \|U_k'; H^1(\varpi)\|
+ C \epsilon \max\limits_{j=0,1} 
\Vert \nabla  X^\epsilon_{j} ; L^2(\varpi) \Vert^{1/2} 
\Vert   U_k'  ; L^2(\varpi) \Vert^{1/2}
\leq C' \epsilon , \nonumber
\end{eqnarray}
since $\Vert \nabla  X^\epsilon_{j} ; L^2(\varpi) \Vert \leq C \epsilon^{1/2}$, due to the
measure of the support of $\nabla  X^\epsilon_{j}$.  \ \ $\Box$
\end{proof}

As a corollary of this lemma, if 
$\epsilon\in (0, \epsilon_0]$, then  the bounds 
\begin{equation}
\label{1.52}
0 < \mu_k \leq c_\mu 
\quad \|{\mathcal{U}}; \cH^{\epsilon} \|\geq c_{\mathcal{U}}>0,
\end{equation}
hold true with some positive constants $c_\mu$ and $c_{\mathcal{U}}$ depending
on $\varpi$ and $\theta$ only.

The next theorem provides quite accurate asymptotic information 
on the eigenvalues of the model problem and in particular justifies the ansatz \eqref{ansatz1}.

\begin{theorem}
\label{th4.5}
For every $k\geq 1$ there exists a constant $C_k$
such that, for each $\eta\in[0,2\pi)$,
\begin{equation}
\label{3.101}
|\Lambda_k^\epsilon(\eta)-\Lambda_k^0-\epsilon \Lambda_k'(\eta)|<C_k\epsilon^{3/2}\,,
\end{equation}
where $\Lambda'_k(\eta)=\pi \capas_3(\theta)|U_k^0({P^0})-e^{-i\eta}U_k^0({P^1})|^2$ (cf.\,\eqref{2.20})
in the case the  eigenvalue  $\Lambda_k^0$ is simple and $\Lambda'_k(\eta) $ is  given by  the formulas \eqref{L-mult-1}--\eqref{L-mult-2} in the case $\Lambda_k^0$ is  a multiple eigenvalue. 
\end{theorem}

\begin{proof}
We apply Lemma \ref{nearEig} to the operator ${\mathcal{B}}^\epsilon (\eta)$
with $\mu = \mu_k $ and ${\mathcal{V}} = {\mathcal{V}}_k$ as in \eqref{use lemma}. Our aim is to show that 
$\tau$ of the lemma can be chosen as small as $C_k \epsilon^{3/2}$. The lemma then gives an eigenvalue $M (\epsilon,\eta)$ of ${\mathcal{B}}^\epsilon (\eta)$ with 
the estimate
\begin{equation}
|M(\epsilon,\eta) - \mu_k | \leq C_k \epsilon^{3/2} . \nonumber
\end{equation}
Using \eqref{use lemma} and \eqref{1.35} this turns into an eigenvalue
$\lambda(\epsilon , \eta)$ (of \eqref{model}--\eqref{quasiperiodic-2}) satisfying \eqref{3.101} in the place of $\Lambda_k^\epsilon(\eta)$. However, if $\epsilon$
is small enough, Theorem \ref{th3.1} guarantees that in a neighbourhood of 
$\Lambda_k^0$ there is only one eigenvalue of the model problem, namely
$\Lambda_k^\epsilon(\eta)$. So $\lambda(\epsilon , \eta)$ must coincide with
it, and the estimate \eqref{3.101} follows.  

We are thus left with the task of proving
\begin{equation*}
\tau=\|{\mathcal{B}}^\epsilon(\eta){\mathcal{V}}_k- \mu_k  {\mathcal{V}}_k; \cH^\epsilon\|\,
\leq C_k \epsilon^{3/2}.
\end{equation*}
To this end  we write, using ${\mathcal{V}}_k = c_{\cU}^{-1} \cU_k $, \eqref{1.26}, \eqref{1.25}, \eqref{use lemma}, \eqref{1.52},
\begin{eqnarray}
\tau& =& \sup_Z 
\big|({\mathcal{B}}^\epsilon(\eta){\mathcal{V}}_k- \mu_k  {\mathcal{V}}_k, Z)_{\epsilon} \big| \nonumber \\
& = & c_{\cU_k}^{-1} \sup_Z \big| ( \cU_k , Z)_\gamma - \mu_k  ({\mathcal{U}}_k, Z)_{\gamma}
- \mu_k ( \nabla \cU_k , \nabla Z)_\varpi \big| \nonumber \\
& \leq  & c_\mu c_{\mathcal{U}}^{-1}
\sup_Z \big| (\Lambda^0_k +
\epsilon \Lambda'_k)({\mathcal{U}}_k, Z)_{\gamma} -(\nabla {\mathcal{U}}_k, \nabla Z)_{\varpi} \big|  =:  c_\mu c_{\mathcal{U}}^{-1}
\sup_{Z}|T(Z)|\, .
\end{eqnarray}
The supremum is calculated here over all functions $Z\in \cH^{\epsilon}$
with unit norm. The  expression $T(Z)$ can be represented as a sum of the terms
\begin{eqnarray*}
S_1(Z)&=&-(\nabla U^0_k,\nabla Z)_\varpi + \Lambda^0_k (U^0_k, Z)_\gamma,  \\
S_{2j}(Z)&=& -\left(\nabla (X^\epsilon_j (U_k^0({P^j})-U_k^0)), \nabla Z\right)_\varpi, \\
S_3(Z)&=& - \big(\nabla \big(\chi_0 w_{k0}(\epsilon^{-1}(x -P^0))+\chi_1 w_{k1}(\epsilon^{-1}(x -P^0))\big),\nabla Z \big)_\varpi \\
& + & \epsilon(\nabla  w_k, \nabla Z)_\varpi,\\
S_4(Z)&=& -\epsilon(\nabla U'_k,\nabla Z)_\varpi -\epsilon(\nabla  w_k, \nabla Z)_\varpi \\
   &+& \epsilon \Lambda'_k (U^0_k, Z)_\gamma+\epsilon \Lambda^0_k (U'_k, Z)_\gamma, \nonumber \\
S_5(Z)&=&-\epsilon(\nabla ((X^\epsilon_0+X^\epsilon_1) U'_k),\nabla Z)_\varpi+\epsilon^2\Lambda'_k(U'_k,Z)_\gamma.
\end{eqnarray*} 
where  $w_k$ is given by
\begin{equation}
\label{3.33}
 w_k(x)=\chi_0(x) \frac{a_{k0}\capas_3(\theta)}{|x-{P^0}|}
 + \chi_1(x)\frac{a_{k1}\capas_3(\theta)}{|x-{P^1}|}\,.
\end{equation}
 
First we note that $S_1(Z)=S_4(Z)=0$, because $U_k^0$ and $U_k'$ are the solutions of the problems \eqref{model-0}--\eqref{model-0-b-cond-2} and \eqref{U'-eq}--\eqref{U'-b-cond-2}, respectively; see also \eqref{2.20ac}.

To estimate $S_{2j}(Z)$ we use the Cauchy-Schwartz inequality:
\begin{eqnarray*}
& & \big| \big(\nabla (X^\epsilon_j (U_k^0({P^j})-U_k^0)), \nabla Z \big)_\varpi \big|
 \\ 
& \leq&
\big\|\nabla(X^\epsilon_j(U_k^0({P^j})-U_k^0)); L^2(\varpi) \big\| \|Z; H^1(\varpi)\| \leq C_k\epsilon^{3/2}.
\end{eqnarray*}
Here the last inequality follows from an estimate already made for 
\eqref{first} (see \eqref{3.46}) and the assumption on the norm of $Z$.

The first term in $S_5(Z)$ can also be treated using the Cauchy-Schwartz inequality,
the properties of the cut-off functions $X^\epsilon_j$ (once again as in \eqref{3.46}) and Lemma \ref{lem2.1}
\begin{eqnarray*}
& & \epsilon \big| \big(\nabla ((X^\epsilon_0+X^\epsilon_1 )U_k' ), \nabla Z \big)_\varpi \big|
 \\ 
& \leq& \epsilon
\big\|\nabla((X^\epsilon_0+X^\epsilon_1 )U_k' ); L^2(\varpi) \big\| \|Z; H^1(\varpi)\|  \\
& \leq &  \epsilon
\big\|\nabla(X^\epsilon_0+X^\epsilon_1 ); L^2(\varpi) \big\| 
\big\|U_k' ; L^2(\varpi) \big\| \\
& \ \ +& \epsilon
\big\|X^\epsilon_0+X^\epsilon_1 ; L^2(\varpi) \big\|
\big\|\nabla U_k' ; L^2(\varpi) \big\|
\leq C_k\epsilon^{3/2}.
\end{eqnarray*}
The surface integral in \(S_5(Z)\) can be estimated simply by the trace inequality. 

To provide an upper bound for $S_3(Z)$ we notice that by \eqref{layer0},
\begin{equation}
\label{3.37}
w_{kj} (\epsilon^{-1}(x-P^j)) 
= \epsilon \frac{a_{kj}\capas_3(\theta)}{|x-{P^j}|}
+ \widetilde w_{kj}(\epsilon^{-1}(x-P^j)) \, , \ j=0,1,
\end{equation}
hence, using \eqref{3.33} we can write
\begin{equation*}
S_3(Z) =-\sum_{j=0}^1 \big( \nabla \big(\chi_j \widetilde{w}_{kj}(\epsilon^{-1}(x-P^j) )\big), \nabla Z\big)_\varpi.
\end{equation*}
After integrating by parts and taking into account that $\widetilde{w}_{kj}$
are harmonic functions we obtain
\begin{eqnarray}
S_3(Z)& = &  \sum_{j=0}^1 \big( \widetilde{w}_{kj}(\epsilon^{-1}(x-P^j))\Delta\chi_j , Z \big)_\varpi \nonumber \\
\label{3.68}
& + & 2 \big( (\nabla\chi_j )\nabla \big( \widetilde{w}_{kj} (\epsilon^{-1}(x-P^j))\big), Z \big)_\varpi.  
\end{eqnarray}

In the second term, the support of the function $\nabla \chi_j$ is contained
in a set $\{ c \leq  |x - P^j| \leq C\} =: {\mathcal S}_j$ for some constants $0 < c < C$, hence, by the estimate \eqref{new2}, 
\begin{equation}
\label{3.57a}
\big| \nabla  \big( \widetilde{w}_{kj}(\epsilon^{-1}(x-P^j)) \big) \big| 
\leq C \epsilon^{-1} \big( \epsilon^{-1} |x - P^j| \big)^{-3} 
= C\epsilon^{2} |x - P^j|^{-3} 
\ \ {\rm for}  \ x \in {\mathcal S}_j .
\end{equation}
Hence,
\begin{eqnarray}
& & \Big| \big( (\nabla\chi_j )\nabla \big( \widetilde{w}_{kj} (\epsilon^{-1}
(x-P^j))\big), Z \big)_\varpi \Big| \nonumber \\
& \leq & \Big( \int_{{\mathcal S}_j} \big| \nabla  \big( \widetilde{w}_{kj}
(\epsilon^{-1}(x-P^j) \big) 
\big|^{2} dx  \Big)^{1/2} \Vert Z ; L^2(\varpi)\Vert 
\leq  C' \epsilon^{2}.
\end{eqnarray}
The first term in \eqref{3.68} is treated with a similar argument,
since ${\mathcal S}_j$ still contains the support of $\Delta \chi_j$ 
and the estimate
\begin{equation*}
\big|  \widetilde{w}_{kj}(\epsilon^{-1}(x-P^j))  \big| 
\leq C\epsilon^{2} |x - P^j|^{-3} 
\ \ {\rm for}  \ x \in {\mathcal S}_j 
\end{equation*}
again holds, by \eqref{new2}.

We thus get the bound 
\[
\tau\leq C_k(\theta) \epsilon^{3/2} . \ \ \Box
\] 
\end{proof}

As a consequence we can now provide the asymptotic widths and positions of the spectral bands.

\begin{theorem}
\label{cor4.6}
Let the index $k$ be such that the  eigenvalue $\Lambda_k^0$ is simple. Then, 
the band $\Upsilon_k^\epsilon$ of the 
continuous spectrum of the problem \eqref{problem}--\eqref{b-cond} has the asymptotic form
\begin{equation*}
\Upsilon_k^\epsilon= [\Lambda_k^0+A_k\epsilon+O(\epsilon^{3/2}), \Lambda_k^0+B_k\epsilon+O(\epsilon^{3/2})]
\end{equation*}
where
\begin{eqnarray}
A_k=\pi \capas_3(\theta) \min\{|U_k^0({P^0})-U_k^0({P^1})|^2, |U_k^0({P^0})+U_k^0({P^1})|^2\}, \\
B_k=\pi \capas_3(\theta) \max\{|U_k^0({P^0})-U_k^0({P^1})|^2, |U_k^0({P^0})+U_k^0({P^1})|^2\}. 
\end{eqnarray}                                                                        
\end{theorem}

\begin{remark} \rm
Returning to the band-gap structure of the Bloch spectrum \eqref{eq1.1}, in general it may happen
that a spectral band, closed interval,  degenerates into a single point. In this case the band consists of a 
single eigenvalue of infinite multiplicity, and the band is thus contained in the essential but not in
the continuous spectrum. However, by Theorem \ref{cor4.6}, if $k$ is such that both $U_k^0({P^0})$ and $U_k^0({P^1})$ are nonzero
and $\epsilon$ is small enough, the numbers $A_k$ and $B_k$ are distinct,
and in this case the spectral band $\Upsilon_k^\epsilon$ is  indeed an interval with positive length.  
This obviously provides a way to construct examples where the spectrum of the linear water-wave problem has a genuine band-gap structure with proper intervals as bands and with at least a given number of  spectral gaps, cf. Theorem \ref{cor3.2}. 
\end{remark}

\newpage
~


\begin{thebibliography}{99}

\bibitem{ADN} Agmon, S., Douglis, A., Nirenberg, L., Estimates near the boundary
for solutions of elliptic differential equations satisfying general
boundary conditions. I., Comm. Pure Appl. Math.  12, 623--727, 1959.

\bibitem{AC} Allaire, G., Conca, C., Bloch wave homogenization and
spectral asymptotic analysis, J. Math. Pures Appl., 77, 153--208, 1998.




\bibitem{BiSo} Birman, M.S., Solomyak, M.Z., Spectral Theory of Self-Adjoint
Operators in Hilbert Space, Reidel Publishing Company, Dordrecht, 1986.

 
 
 
\bibitem{CaMI} Carter, B.G., McIver, P., Water-wave propagation through an infinite
array of floating structures. Journal of Engineering
Mathematics, 2012 




\bibitem{Kato} Kato, T.,  Perturbation theory for linear operators. Die
Grundlehren der mathematischen Wissenschaften, Band 132
Springer-Verlag New York, 1966.

\bibitem{KoMaRo} Kozlov, V.A., Mazja, V.G., Rossmann, J., 
Elliptic boundary value problems in domains with point singularities,
American Mathematical Soc., Providence, 1997


\bibitem{Kuch} Kuchment, P., Floquet theory for partial differential equations. Operator Theory: Advances and Applications, 60, Birkhäuser Verlag, Basel, 1993. 


 \bibitem{lin} Linton, C. M.,  Water waves over arrays of horizontal cylinders: band
gaps and Bragg resonance. Journal of Fluid Mechanics, 670, 504-526, 2011.


\bibitem{Liu} Liu, P.L.-F., Resonant reflection of water waves in a long channel with corrugated
boundaries, J. Fluid. Mech., 245, 371-381, 1987.

 
\bibitem{MaNaPl} 
Mazya V. G., Nazarov S. A. and Plamenevskii B. A., 
Asymptotic theory of elliptic boundary value 
problems in singularly perturbed domains, 
Birkh\"{a}user Verlag, Basel, 2000.

\bibitem{Mat} Mattioli, F., Resonant reflection of a series of submerged breakwaters, Il Nuovo
Cimento C \textbf{13} C, pp. 823-833, 1990.

\bibitem{McI} McIver, P.,  Water-wave propagation through an infinite array of
cylindrical structures. Journal of Fluid Mechanics, 424, 101-125, 2000.

\bibitem{McKee} McKee, W.D., The propagation of water waves along a channel of variable
width, Applied Ocean Research, 21, 145-156, 1999.

\bibitem{Mei} Mei, C.C., Resonant reflection of surface water waves by periodic sandbars, J. Fluid
Mech.,  152,  315-335, 1985.


\bibitem{na17} Nazarov S.A., Elliptic boundary value problems with periodic
coefficients in a cylinder,  Izv. Akad. Nauk SSSR. Ser. Mat.
45 (1) 101-112, 1981. (English transl.: Math. USSR. Izvestija.
18 (1), 89-98, 1982)


\bibitem{NaSpec} Nazarov, S.A., Properties of spectra of boundary value
problems in cylindrical and quasicylindrical domains,
Sobolev Spaces in Mathematics, vol. II  (Maz'ya V., Ed.)
International Mathematical Series 9, 261--309, 2008.




\bibitem{na460} Nazarov S.A., Opening gaps in the spectrum of the water-wave
problem in a periodic channel, Zh. Vychisl. Mat. i Mat. Fiz.,
50, 6, 1092--1108, 2010 (English transl.: Comput. Math. and Math.
Physics 50, 6, 1038--1054, 2010).

\bibitem{NaPl} Nazarov, S.A, Plamenevskii, B.A, Elliptic problems in domains
with piecewise smooth boundaries, Walter be Gruyter, Berlin, New York, 1994.

\bibitem{NaRuTa} Nazarov, S.A., Ruotsalainen, K., Taskinen, J., Essential spectrum of a periodic elastic waveguide may contain arbitrarily many gaps, Appl. Anal. 89,1, 109-124, 2010.

\bibitem{NaTa} Nazarov, S.A.,  Taskinen, J.,  On essential and continuous spectra of the
linearized water-wave problem in a finite pond, Math. Scand. 106, 1, 141-160, 2010.

\bibitem{PolyaSzego} 
P\'{o}lya G. and Szeg\"{o} G.; 
Isoperimetric inequalities in mathematical physics, 
Princeton University Press, N.J., 1951.

\bibitem{LordRa} Lord Rayleigh; On the maintenance of vibrations by forces of double frequency, and
on the propagation of waves through a medium endowed with a periodic structure, Philos. Mag., 24, 145-159, 1887.

\bibitem{ViLu} 
Visik M. I. and Ljusternik L. A.; 
Regular degeneration and boundary layer of linear differential 
equations with small parameter, 
Amer. Math. Soc. Transl. 20, 239-364, 1962.

 


 


 


\end{thebibliography}
\end{document}